\documentclass[12pt,a4paper]{amsart}
\usepackage{amsfonts}
\usepackage{ifthen}
\usepackage{amsmath}
\usepackage{verbatim}
\usepackage{graphicx}
\usepackage{amscd,amssymb,amsthm}
\newcounter{minutes}
\divide\time by 60
\newcounter{hours}
\multiply\time by 60 \addtocounter{minutes}{-\time}
\newcommand{\K}{\mathcal{K}}
\newcommand{\E}{\mathcal{E}}

\newcommand{\R}{\mathbb{R}}
\newcommand{\HH}{\mathbb{H}}
\newcommand{\C}{\mathbb{C}}
\newcommand{\B}{\mathbb{B}}

\newcommand{\beq}{\begin{equation}}
\newcommand{\eeq}{\end{equation}}

\DeclareMathOperator{\ch}{ch}
\DeclareMathOperator{\sh}{sh}

\DeclareMathOperator{\arctanh}{arth}

\date{}

\title{\bf  On exterior moduli of quadrilaterals and special functions}

\author{Matti Vuorinen}
\author{Xiaohui Zhang }

\address{Department of Mathematics and Statistics, University of Turku, 20014 Turku,
Finland} \email{vuorinen@utu.fi, xiazha@utu.fi}

\newtheorem{theorem}[equation]{Theorem}

\newtheorem{lemma}[equation]{Lemma}

\newtheorem{corollary}[equation]{Corollary}
\newtheorem{openprob}[equation]{Open problem}

\numberwithin{equation}{section}

\theoremstyle{remark}
\newtheorem{remark}[equation]{Remark}
\newtheorem{nonsec}[equation]{}
\begin{document}

%===============================================================================
\begin{abstract}
In this paper two identities involving a function defined by the complete elliptic integrals of the first and second kinds are proved. Some functional inequalities and elementary estimates for this function are also derived from the properties of monotonicity and convexity of this function. As applications, some functional inequalities and the growth of the exterior modulus of a rectangle are studied.
\end{abstract}

%===============================================================================
%%%%%%%% BEGIN TIMESTAMP
\def\thefootnote{}
\footnotetext{ \texttt{\tiny File:~\jobname .tex,
          printed: \number\year-\number\month-\number\day,
          \thehours.\ifnum\theminutes<10{0}\fi\theminutes}
} \makeatletter\def\thefootnote{\@arabic\c@footnote}\makeatother
%%%%%%%% END TIMESTAMP
%===============================================================================

%%%%%%%%%%%%%%%%%%

\maketitle



%===============================================================================
%===============================================================================
%===============================================================================
\medskip

\section{Introduction}
%===============================================================================
\begin{nonsec}{\bf Exterior modulus of a quadrilateral.}\,
For $h>0$ consider the rectangle $D$ with vertices $1+ih$, $ih$, $0$, $1$ in the upper half plane $\HH^2=\{x+iy:y>0\}$ and a bounded harmonic function $u:\C\setminus D\to\R$ satisfying the Dirichlet-Neumann boundary value problem $u(z)=0$ for $z\in[0,1]$, $u(z)=1$ for $z\in[ih,1+ih]$, $\frac{\partial u}{\partial n}(z)=0$ for $z\in[1,1+ih]\cup[0,ih]$ where $n$ is the direction of the exterior normal to $\partial D$. The number
$$
  \mathcal{M}(1+ih,ih,0,1)=\int_{\C\setminus D}|\nabla u|^2dm
$$
is called the exterior modulus of the rectangle $D(1+ih,ih,0,1)$.

This quantity also has an interpretation as the modulus of the family of all curves, joining the segments $[1+ih,ih]$ and $[0,1]$ in the complement of the rectangle $D$, which also is equal to $\mathcal{M}(1+ih,ih,0,1)$ (cf. \cite{ah}).
For a polygonal quadrilateral $D(a,b,0,1)$ with vertices $a,b\in\HH^2$ and base $[0,1]$, the exterior modulus $\mathcal{M}(a,b,0,1)$ can be defined in the same way.

As far as we know there is no analytic formula for $\mathcal{M}(a,b,0,1)$. Numerical methods for the computation of $\mathcal{M}(a,b,0,1)$ were recently studied by H. Hakula, A. Rasila, and M. Vuorinen in \cite{hrv} which motivate the present study. They used numerical methods such as hp-FEM and the Schwarz-Christoffel mapping. Similar problems for the interior modulus have been studied in \cite{hqr, hrv1}. The literature and software dealing with numerical conformal mapping problems are very wide, see e.g. \cite{dt, ps}.

Here we study the above problem for the case of a rectangle. In this case an explicit formula involving complete elliptic integrals was given by P. Duren and J. Pfaltzgraff \cite{dp}, and our goal is to analytically study the dependence of the formula on $h$.
\end{nonsec}

\begin{nonsec}{\bf Complete elliptic integrals.}\,
Let $\K(r)$ and $\E(r)$ stand for the complete elliptic integrals of the first and second kind, respectively (see (\ref{ellint})). Let $r'=\sqrt{1-r^2}$ for $r\in(0,1)$. We often denote
$\K'(r)=\K(r'),\quad \E'(r)=\E(r')$. Define the function $\psi$ as follows
\beq\label{psi}
\psi(r)=\frac{2(\E(r)-(1-r)\K(r))}{\E'(r)-r\K'(r)},\quad r\in(0,1).
\eeq
The function $\psi:(0,1)\to(0,\infty)$ is a homeomorphism, see Theorem \ref{mythm1} or \cite{dp}. In particular, $\psi^{-1}:(0,\infty)\to(0,1)$ is well-defined.
\end{nonsec}

\begin{nonsec}{\bf Duren-Pfaltzgraff formula for a rectangle.}\,
In \cite{dp}, P. Duren and J. Pfaltzgraff studied the modulus $\mathcal{M}(\Gamma)$ of the family of curves $\Gamma$ joining the opposite sides of length $b$ of the rectangle with sides $a$ and $b$, in the exterior of the rectangle, and gave the formula \cite[Theorem 5]{dp}
\beq\label{dpformula}
\mathcal{M}(\Gamma)=\dfrac{\K'(r)}{2\K(r)},\quad \mbox{where}\quad r=\psi^{-1}(a/b).
\eeq
The exterior modulus $\mathcal{M}(\Gamma)$ is a conformal invariant of a quadrilateral. In \cite{adv}, the authors gave a sharp comparison between the function $\psi$ and Robin modulus of a given rectangle. Their result can be rewritten as the following inequality
\beq\label{advinequal}
\dfrac{\pi r}{(1-r)^2}<\psi(r)<\dfrac{16r}{\pi(1-r)^2},\quad r\in(0,1).
\eeq
\end{nonsec}

In this paper two identities involving the function $\psi$ are proved, and some functional inequalities and elementary estimates for the function $\psi$ are also derived from the monotonicity and convexity of the combinations of the function $\psi$ and some elementary functions. As applications, we will study the growth of the exterior modulus with respect to the length of one side of the rectangle. The main results are listed as follows.

\begin{theorem} \label{myidentity}
For $r\in(0,1)$, the function $\psi$ satisfies the identities
$$\psi(r^2)\psi\left(\left(\dfrac{1-r}{1+r}\right)^2\right)=1,\quad \psi\left(\dfrac{1-r}{1+r}\right)\psi\left(\dfrac{1-r'}{1+r'}\right)=1.$$
\end{theorem}

\begin{theorem} \label{mythm2}
The function $f(r)=(1-\sqrt{r})^2\psi(r)/r$ is strictly decreasing from $(0,1)$ onto $(4/\pi,\pi)$. In particular, for all $r\in(0,1)$
$$\dfrac{4r}{\pi(1-\sqrt{r})^2}<\psi(r)<\dfrac{\pi r}{(1-\sqrt{r})^2}.$$
\end{theorem}

\begin{theorem} \label{mythm3}
The function $f(x)=\psi(1/\ch(x))$ is decreasing and convex from $(0,\infty)$ onto $(0,\infty)$. In particular, for $r,s\in(0,1)$,
\beq\label{myfunineq}
2\psi\left(\frac{\sqrt{2rs}}{\sqrt{1+rs+r's'}}\right)\leq\psi(r)+\psi(s)
\eeq
with equality in the above inequality if and only if $r=s$.
\end{theorem}

\begin{theorem} \label{mythm5}
For $x,y\in(0,1)$,
$$\psi\left(H_p(x,y)\right)\leq H_p\left(\psi(x),\psi(y)\right) {\quad} \mbox{if}{\quad}p\geq0,$$
and
$$\psi\left(H_p(x,y)\right)\geq H_p\left(\psi(x),\psi(y)\right) {\quad} \mbox{if}{\quad} p\leq-1.$$
The equality holds in each case if and only if $x=y$. Here $H_p$ is the power mean defined as
$$
H_p(x,y)=\left\{
\begin{array}{ll}
\left(\dfrac{x^p+y^p}{2}\right)^{1/p},&p\neq0\\
\sqrt{xy},&p=0.
\end{array}
\right.
$$
\end{theorem}

%===============================================================================
%===============================================================================
\medskip

\section{Preliminaries}
%===============================================================================
For $0<r<1$, the functions
\beq\label{ellint}
\K(r)=\int_0^{\pi/2}\dfrac{dt}{\sqrt{1-r^2\sin^2t}}, \quad
\E(r)=\int_0^{\pi/2}\sqrt{1-r^2\sin^2t}\,dt
\eeq
with limiting values $\K(0)=\pi/2=\E(0)$, $\K(1-)=\infty$ and $\E(1)=1$
are known as Legendre's complete elliptic integrals of the first and second kind, respectively.
These two functions are connected by Legendre's relation \cite[110.10]{bf}
\beq\label{legendre}
\E\K'+\E'\K-\K\K'=\dfrac\pi2.
\eeq
Some derivative formulas involving these elliptic integrals are as follows \cite[p.21]{bo}:
\beq\label{derivative}
\left\{\begin{array}{ll}
\dfrac{d\K}{dr}=\dfrac{\E-r'^2\K}{rr'^2},&\dfrac{d\E}{dr}=\dfrac{\E-\K}{r},\vspace{1mm}\\
\dfrac{d}{dr}\left(\E-r'^2\K\right)=r\K, &\dfrac{d}{dr}\left(\K-\E\right)=\dfrac{r\E}{r'^2}.\\
\end{array}\right.
\eeq

The functions $\K$ and $\E$ satisfy the following identities due to Landen \cite[163.01, 164.02]{bf}
\beq\label{transfk1}
\K\left(\dfrac{2\sqrt{r}}{1+r}\right)=(1+r)\K(r),
\eeq
\beq\label{transfk2}
\K\left(\dfrac{1-r}{1+r}\right)=\dfrac12(1+r)\K'(r),
\eeq
\beq\label{transfe1}
\E\left(\dfrac{2\sqrt{r}}{1+r}\right)=\dfrac{2\E(r)-r'^2\K(r)}{1+r},
\eeq
\beq\label{transfe2}
\E\left(\dfrac{1-r}{1+r}\right)=\dfrac{\E'(r)+r\K'(r)}{1+r}.
\eeq

Using Landen's transformation formulas, we have the following identities.

\begin{lemma} \label{mytransf}
For $r\in(0,1)$, let $t=(1-r)/(1+r)$. Then
\beq\label{mytransfk}
\K(t^2)=\dfrac{(1+r)^2}4\K'(r^2),
\eeq
\beq\label{mytransfkk}
\K'(t^2)={(1+r)^2}\K(r^2),
\eeq
\beq\label{mytransfe}
\E(t^2)=\dfrac{\E'(r^2)+(r+r^2+r^3)\K'(r^2)}{(1+r)^2},
\eeq
\beq\label{mytransfee}
\E'(t^2)=\dfrac{4\E(r^2)-(3-2r^2-r^4)\K(r^2)}{(1+r)^2}.
\eeq
\end{lemma}

\begin{proof}
By Landen's transformations (\ref{transfk1}) and (\ref{transfk2}), we have
$$\dfrac{2(1+r^2)}{(1+r)^2}\K(t^2)=\K\left(\dfrac{1-r^2}{1+r^2}\right)=\dfrac12(1+r^2)\K'(r^2).$$
This implies (\ref{mytransfk}).

For (\ref{mytransfkk}),
$$
\K'(t^2)=\dfrac{(1+r)^2}{1+r^2}\K\left(\dfrac{2r}{1+r^2}\right)=(1+r)^2\K(r^2)
$$
where the first equality is Landen's transformation (\ref{transfk2}) with the parameter $t^2$ and the second equality follows from (\ref{transfk1}) with  the parameter $r^2$.

Using Landen's transformation (\ref{transfe1}) with the change of parameter $r\mapsto t^2$ and the formula (\ref{mytransfk}), we get
\beq\label{idt1}
\E\left(\dfrac{1-r^2}{1+r^2}\right)=\dfrac{(1+r)^2\E(t^2)-r(1+r^2)\K'(r^2)}{1+r^2}.
\eeq
On the other hand, by (\ref{transfe2})
\beq\label{idt2}
\E\left(\dfrac{1-r^2}{1+r^2}\right)=\dfrac{\E'(r^2)+r^2\K'(r^2)}{1+r^2}.
\eeq
Hence (\ref{mytransfe}) follows from (\ref{idt1}) and (\ref{idt2}).

For (\ref{mytransfee}), by the change of parameter $r\mapsto t^2$ in Landen's transformation (\ref{transfe2}) and the formula (\ref{mytransfkk}), we have
\beq\label{idt3}
\E\left(\dfrac{2r}{1+r^2}\right)=\dfrac{(1+r)^2\E'(t^2)+(1-r^2)^2\K(r^2)}{2(1+r^2)}.
\eeq
On the other hand, by (\ref{transfe1})
\beq\label{idt4}
\E\left(\dfrac{2r}{1+r^2}\right)=\dfrac{2\E(r^2)-(1-r^4)\K(r^2)}{1+r^2}.
\eeq
Hence (\ref{mytransfee}) follows from (\ref{idt3}) and (\ref{idt4}).
\end{proof}

The next lemma is a monotone form of l'H\^opital's rule and will be useful in deriving monotonicity properties and obtaining inequalities \cite[Theorem 1.25]{avv}.

\begin{lemma}[\bf Monotone form of l'H\^opital's Rule]\label{MLR}
Let $-\infty<a<b<\infty$, and let $f,g:[a,b]\to\R$ be continuous on $[a,b]$, differentiable on $(a,b)$. Let $g'(x)\neq0$ on $(a,b)$. Then, if $f'(x)/g'(x)$ is increasing (decreasing) on $(a,b)$, so are
$$\dfrac{f(x)-f(a)}{g(x)-g(a)}\qquad \mbox{and}\qquad \dfrac{f(x)-f(b)}{g(x)-g(b)}.$$
If $f'(x)/g'(x)$ is strictly monotone, then the monotonicity on the conclusion is also strict.
\end{lemma}

The following Lemma \ref{lemma} is from \cite[Theorem 3.21 (1),(7)]{avv}.

\begin{lemma} \label{lemma}
\emph{(1)}\, $r^{-2}(\E-r'^2\K)$ is strictly increasing and convex from $(0,1)$ onto $(\pi/4,1)$.\\
\emph{(2)}\,  For each $c\in[1/2,\infty)$, $r'^c\K$ is decreasing from $[0,1)$ onto $(0,\pi/2]$.
\end{lemma}

\begin{lemma} \label{mylemma}
\emph{(1)}\, $f_1(r)=\E-(1-r)\K$ is strictly increasing and concave from $(0,1)$ onto $(0,1)$.\\
\emph{(2)}\, $f_2(r)=(\E-(1-r)\K)/r$ is strictly decreasing from $(0,1)$ onto $(1,\pi/2)$.\\
\emph{(3)}\, $f_3(r)=\E'-r\K'$ is strictly decreasing and convex from $(0,1)$ onto $(0,1)$.\\
\emph{(4)}\, $f_4(r)=(\E'-r\K')/(1-r)$ is strictly decreasing from $(0,1)$ onto $(0,1)$.\\
\emph{(5)}\, $f_5(r)=(\E-r'\K)/(1-\sqrt{r'})^2$ is strictly decreasing from $(0,1)$ onto $(1,\pi/2)$.\\
\emph{(6)}\, $f_6(r)=(3-r)\E'-(1+r)\K'$ is increasing form $(0,1)$ onto $(-\infty,0)$.\\
\emph{(7)}\, $f_7(r)=(1+r)(\E'-r\K')/(1-r)$ is strictly deceasing from $(0,1)$ onto $(0,1)$.\\
\emph{(8)}\, $f_8(r)=(\E-(1-r)\K)/(\sqrt{r}(1-r)\K)$ is strictly increasing from $(0,1)$ onto $(0,\infty)$. $f_8(0.479047\cdots)=1.$
\end{lemma}

\begin{proof}
(1)\, By differentiation and the derivative formulas (\ref{derivative}),
$$f_1'(r)=\dfrac{\E}{1+r}$$
which is positive and decreasing. Then the properties of monotonicity and concavity of $f_1$ follow. The limiting value
$f_1(0)=0$ is clear and $f_1(1-)=\E(1)-\lim\limits_{r\to1-}(r'^2\K/(1+r))=1$ by Lemma \ref{lemma}(2).

(2)\,  Since $f_1$ is concave and $f_2(r)=f_1(r)/r$, $f_2$ is decreasing by the monotone form of
l'H\^opital's rule. By l'H\^opital's rule $f_2(0)=f'_1(0)=\pi/2$, and $f_2(1)=1$ is clear.

(3)\, By (\ref{derivative}), we have
$$f_3'(r)=-\dfrac{\K'-\E'}{1+r},$$
which is negative and increasing in $(0,1)$. Then $f_3$ is decreasing and convex in $(0,1)$. The limiting value
$f_3(0)=1$ follows from Lemma \ref{lemma}(2), and $f_3(1)=0$ is clear.

(4)\, Let $h(r)=1-r$. Since $f_3$ is convex, $f_3'(r)/h'(r)$ is decreasing. Thus $f_4(r)=f_3(r)/h(r)$ is also decreasing by the monotone form of l'H\^opital's rule. By l'H\^opital's rule $f_4(1)=-f'_3(1)=0$, and $f_4(0)=f_3(0)=1$.

(5)\, For the proof we first make the change of variable $r=2\sqrt{x}/(1+x)$. The Landen transformations (\ref{transfk1}) and (\ref{transfe1}) lead to
$$h(x)=f_5\left(\dfrac{2\sqrt{x}}{1+x}\right)=\dfrac{\E(x)-x'^2\K(x)}{1-x'}=\dfrac{h_1(x)}{h_2(x)},$$
where $h_1(x)=\E(x)-x'^2\K(x)$ and $h_2(x)=1-x'$ with $h_1(0)=0=h_2(0)$. Then by (\ref{derivative}) we have
$$\dfrac{h_1'(x)}{h_2'(x)}=\dfrac{x\K(x)}{x/x'}=x'\K(x),$$
which is strictly decreasing by Lemma \ref{lemma}(2). This implies that $h$ is decreasing by the monotone form of
l'H\^opital's rule, and hence $f_5$ is also decreasing in $(0,1)$.

(6)\, By differentiation, we have
$$f_6'(r)=\dfrac{(1-r)(2r\E'+\E'-r\K')}{r(1+r)}>0,$$
and hence $f_6$ is increasing. The limiting values are clear.

(7)\, By simple computation, $f_7'(r)=f_6(r)/(1-r)^2<0$ and hence $f_7$ is decreasing. The limiting values follow from part (4).

(8)\, Differentiation and simplification give that
$$f_8'(r)=\dfrac{((1+r)\K-\E)(2\E-r'^2\K)}{2r^{3/2}(1+r)(1-r)^2\K^2}>0,$$
and hence $f_8$ is strictly increasing. The limit
$$f_8(0+)=\lim_{r\to0+}\dfrac{\E-(1-r)\K}{\sqrt{r}\K}=\lim_{r\to0+}\dfrac{\E-(1-r)\K}{r}\dfrac{\sqrt{r}}{\K}=0$$
follows from the part (2). The limit $f_8(1-)=\infty$ is clear.
\end{proof}

Let
\begin{equation}\label{grotzsch}
\mu(r)=\frac{\pi}{2} \frac{\K'(r)}{\K(r)}
\end{equation}
be the modulus of Gr\"otzsch's ring $\B^2\setminus[0,r]$ (see \cite{lv},\cite{avv}).

\begin{lemma} \label{mylemmupsi}
The function $f(r)=\mu(r)\psi(r)$ is strictly increasing from $(0,1)$ onto $(0,\infty)$.
\end{lemma}

\begin{proof}
Since the function $f$ can be rewritten as
$$f(r)=\pi\sqrt{r}\K'\dfrac{\E-(1-r)\K}{\sqrt{r}(1-r)\K}\dfrac{1-r}{\E'-r\K'},$$
the conclusion follows from Lemma \ref{lemma}(2), Lemma \ref{mylemma}(4) and (8).
\end{proof}

%===============================================================================
%===============================================================================
\medskip

\section{Proofs of Main Results}
%===============================================================================

In this section  we will prove two identities involving the function $\psi$, and some functional inequalities and elementary estimates for the function $\psi$ are also derived from the monotonicity and convexity of the combinations of the function $\psi$ and some elementary functions.

\begin{theorem} \label{mythm1}
The function $\psi(r)$ is strictly increasing and convex from $(0,1)$ onto $(0,\infty)$, and the function
$\psi(r)/r$ is strictly increasing from $(0,1)$ onto $(0,\infty)$.
\end{theorem}

\begin{proof}
By differentiation, and using (\ref{derivative}) and Legendre's identity (\ref{legendre}), we have
\beq\label{psider}
\dfrac{d\psi}{dr}=\dfrac{2(1-r)}{1+r}\dfrac{\E'\K+\E\K'-\K\K'}{(\E'-r\K')^2}=\dfrac{\pi}{1-r^2}\left(\frac{1-r}{\E'-r\K'}\right)^2,
\eeq
which is positive and strictly increasing by Lemma \ref{mylemma}(4). Hence $\psi(r)$ is strictly increasing and convex, and consequently $\psi(r)/r$ is strictly increasing by the monotone form of l'H\^opital's rule.
\end{proof}

\begin{proof}[\bf Proof of Theorem \ref{myidentity}]
By simple calculations, the first identity follows from the definition of $\psi$ and Lemma \ref{mytransf}.
The second identity follows from the first one with the change of parameter $r\mapsto\sqrt{(1-r)/(1+r)}$.
\end{proof}

\begin{corollary}\label{specvalu}
$\psi(3-2\sqrt{2})=1.$
\end{corollary}

\begin{proof}
Let $r=1/\sqrt{2}$. Then $(1-r)/(1+r)=3-2\sqrt{2}=(1-r')/(1+r')$, and the second identity in the Theorem \ref{myidentity} implies $\psi(3-2\sqrt{2})=1.$
\end{proof}

\begin{remark}
Let $\Delta$ be the family of curves lying outside the rectangle $R$ and joining the opposite sides of length $a$. Then a basic fact is
$$\mathcal{M}(\Gamma)=1/\mathcal{M}(\Delta).$$
By \eqref{dpformula} and \eqref{grotzsch}, we have
$$\mathcal{M}(\Gamma)=\mu(r)/\pi,\quad\mbox{and}\quad  \mathcal{M}(\Delta)=\mu(s)/\pi$$
with $r=\psi^{-1}(a/b)$ and $s=\psi^{-1}(b/a).$ By the identity \cite[Exercises 5.68(2)]{avv}
$$\mu(r^2)\mu\left(\left(\dfrac{1-r}{1+r}\right)^2\right)=\pi^2,$$
it is easy to see that $\mathcal{M}(\Gamma)=1/\mathcal{M}(\Delta)$ is equivalent to $s=((1-\sqrt{r})/(1+\sqrt{r}))^2.$
Since $s=\psi^{-1}(b/a)=\psi^{-1}(1/\psi(r))$, we have $s=((1-\sqrt{r})/(1+\sqrt{r}))^2$ which is equivalent to
$$\dfrac{1}{\psi(r)}=\psi\left(\left(\dfrac{1-\sqrt{r}}{1+\sqrt{r}}\right)^2\right).$$
\end{remark}

\begin{proof}[\bf Proof of Theorem \ref{mythm2}]
The theorem follows from $$f(r)=\dfrac{(1-\sqrt{r})^2\psi(r)}{r}=2\dfrac{\E-(1-r)\K}{r}\dfrac{(1-\sqrt{r})^2}{\E'-r\K'},$$
since $(\E-(1-r)\K)/r$ and $(1-\sqrt{r})^2/(\E'-r\K')$ are both decreasing by Lemma \ref{mylemma}(2), (5), respectively. The limiting values are clear by Lemma \ref{mylemma}(2), (5).
\end{proof}

\begin{corollary}
The function $g(r)=(1-\sqrt{r})\arctanh(1-\sqrt{r})\psi(r)/r$ is strictly decreasing from $(0,1)$ onto $(4/\pi,\infty)$.
\end{corollary}

\begin{proof}
This follows from Theorem \ref{mythm2}, since $g(r)=f(r)\arctanh(1-\sqrt{r})/(1-\sqrt{r})$ where $f(r)$ is as in Theorem \ref{mythm2} and $\arctanh(1-\sqrt{r})/(1-\sqrt{r})$ is strictly decreasing from $(0,1)$ onto $(1,\infty)$.
\end{proof}

Since the bounds for $\psi$ in (\ref{advinequal}) and the Theorem \ref{mythm2} are not comparable in the whole interval $(0,1)$, we could combine them to get the following inequalities:

\begin{corollary}
For $0<r<1$,
$$\max\left\{\dfrac{\pi r}{(1-r)^2},\dfrac{4r}{\pi(1-\sqrt{r})^2}\right\}<\psi(r)<\min\left\{\dfrac{16r}{\pi(1-r)^2},\dfrac{\pi r}{(1-\sqrt{r})^2}\right\}.$$
\end{corollary}

\medskip
\begin{figure}[h]
\begin{minipage}[t]{0.45\linewidth}
\centering
\includegraphics[width=7cm]{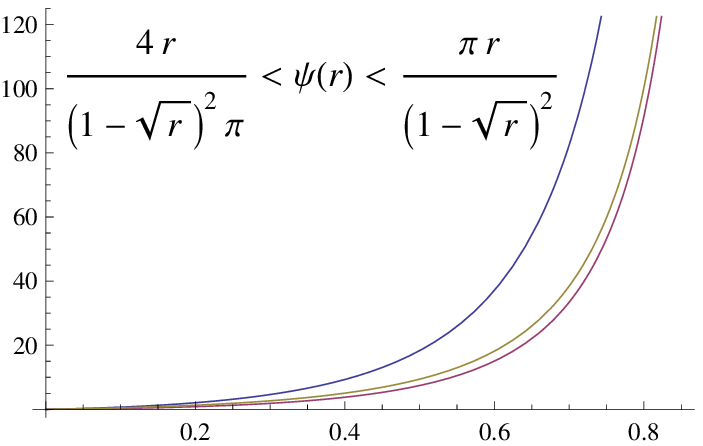}
%\caption{ \label{h2}}
\end{minipage}
%\hfill
\hspace{0.5cm}
\begin{minipage}[t]{0.45\linewidth}
\centering
\includegraphics[width=7cm]{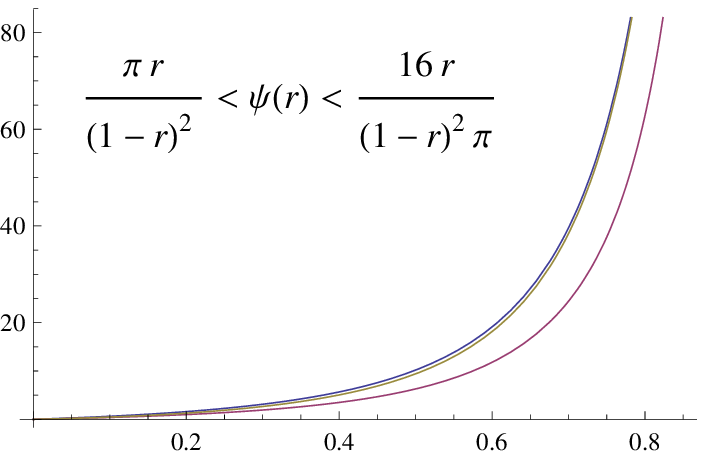}
%\caption{ \label{b2}}
\end{minipage}
\end{figure}

\medskip

\begin{figure}[h]
\begin{minipage}[t]{0.45\linewidth}
\centering
\includegraphics[width=6.8cm]{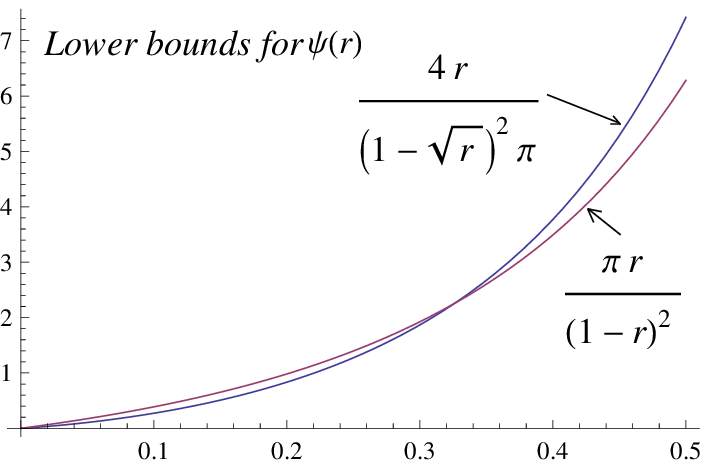}
%\caption{ \label{h2}}
\end{minipage}
%\hfill
\hspace{0.5cm}
\begin{minipage}[t]{0.45\linewidth}
\centering
\includegraphics[width=7cm]{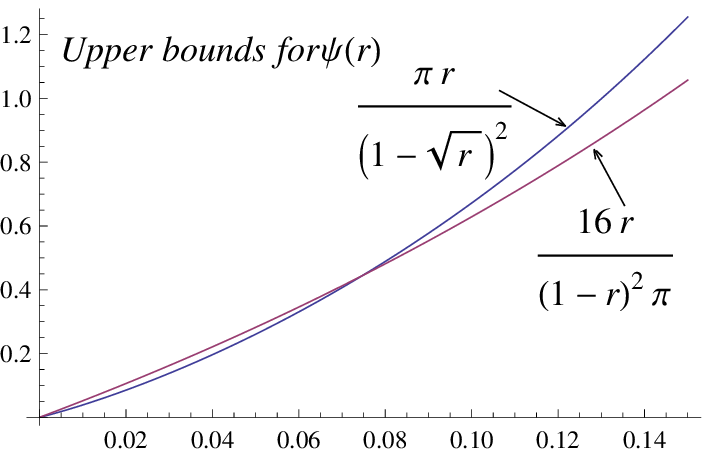}
%\caption{ \label{b2}}
\end{minipage}
\end{figure}
\medskip

\begin{proof}[\bf Proof of Theorem \ref{mythm3}]
Let $r=1/\ch(x)$ and $s=1/\ch(y)$. Then $dr/dx=-\sh(x)/\ch^2(x)=-rr'$ and
$$f'(x)=-\pi rr'\dfrac{1-r}{1+r}\dfrac{1}{(\E'-r\K')^2}=-\pi g(r),$$
where $g(r)=rr'(1-r)/((1+r)(\E'-r\K')^2)$. By the change of variable $r=(1-t)/(1+t)$ and using Landen's transformations (\ref{transfk2}) and (\ref{transfe2}), we have
$$g\left(\dfrac{1-t}{1+t}\right)=\dfrac{1-t}{2}\frac{t^{3/2}}{(\E-t'^2\K)^2}$$
which is decreasing in $t$ by Lemma \ref{lemma}(1), and consequently, $f'(x)$ is increasing in $x$. Therefore, $f$ is decreasing and convex on $(0,\infty)$. In particular, we have $f((x+y)/2)\leq(f(x)+f(y))/2$, with equality if and only if  $x=y$. Now
$$\ch^2\left(\dfrac{x+y}{2}\right)=\dfrac{1+rs+r's'}{2rs}.$$
Hence
$$f\left(\frac{x+y}2\right)\leq\frac{f(x)+f(y)}2$$
gives
$$\psi(r)+\psi(s)\geq2\psi\left(\frac{\sqrt{2rs}}{\sqrt{1+rs+r's'}}\right),$$
with equality if and only if $r=s$.
\end{proof}

\begin{remark}
It is clear that $f(x)$ is decreasing and $2f(x+y)\leq f(x)+f(y)$.
Since $$\ch(x+y)=\dfrac{1+r's'}{rs},$$
we have
$$2\psi\left(\dfrac{rs}{1+r's'}\right)\leq\psi(r)+\psi(s)$$
which is weaker than the inequality (\ref{myfunineq}).
\end{remark}

A function $f: I\to J$ is called $H_{p,q}-$convex (concave) if it satisfies
$$f(H_{p}(x,y))\leq(\geq)H_q(f(x),f(y))$$
for all $x,y\in I$ and strictly $H_{p,q}-$convex (concave) if the inequality is strict, except for $x=y$. Recently, many authors investigated the $H_{p,q}-$convexity (concavity) of special functions, see \cite{avv2, balpv, ba, bapv, cwzq, wzj}. The following theorems give some functional inequalities by studying the generalized convexity (concavity) of the function $\psi$.

\begin{theorem} \label{mythm4}
The function $f(x)=\log(1/\psi(e^{-x}))$ is strictly increasing and concave from $(0,\infty)$ onto $(-\infty,\infty)$. In particular, for $r,s\in(0,1)$,
$$\psi(\sqrt{rs})\leq\sqrt{\psi(r)\psi(s)}$$
with equality if and only if $r=s$.
\end{theorem}

\begin{proof}
Let $r=e^{-x}$ and $s=e^{-y}$. Then $dr/dx=-r$ and
$$f'(x)=\dfrac{r\psi'(r)}{\psi(r)}=\frac{\pi}{2}\dfrac{r}{\E-(1-r)\K}\dfrac{1-r}{(1+r)(\E'-r\K')}$$
which is positive and increasing in $r$ by Lemma \ref{mylemma}(2) and (7), hence decreasing in $x$.
Therefore, $f$ is strictly increasing and concave on $(0,\infty)$. In particular, we have $f((x+y)/2)\geq(f(x)+f(y))/2$, with equality if and only if  $x=y$. This gives
$$\psi(\sqrt{rs})\leq\sqrt{\psi(r)\psi(s)}$$
with equality if and only if $r=s$.
\end{proof}

\begin{proof}[\bf Proof of Theorem \ref{mythm5}]
For $p=0$, the inequality is from Theorem \ref{mythm4}. Now we assume that $p\neq0$.
Let $0<x<y<1$ and $t=((x^p+y^p)/2)^{1/p}>x$. Define
$$f(x)=\psi(t)^p-\dfrac{\psi(x)^p+\psi(y)^p}{2}.$$
By differentiation, we have $dt/dx=\frac{1}{2}(x/t)^{p-1}$ and
\begin{eqnarray*}
f'(x)&=&\dfrac12p\psi(t)^{p-1}\psi'(t)\left(\dfrac{x}{t}\right)^{p-1}-\dfrac12p\psi(x)^{p-1}\psi'(x)\\
&=&\dfrac{p}2x^{p-1}\left(\left(\dfrac{\psi(t)}{t}\right)^{p-1}\psi'(t)-\left(\dfrac{\psi(x)}{x}\right)^{p-1}\psi'(x)\right).
\end{eqnarray*}

We first consider the case of $p>0$. Previous calculation gives
\begin{eqnarray*}
f'(x)&=&\dfrac{p}2x^{p-1}\left(\left(\dfrac{\psi(t)}{t}\right)^{-1}\psi'(t)\left(\dfrac{\psi(t)}{t}\right)^{p}-\left(\dfrac{\psi(x)}{x}\right)^{-1}\psi'(x)\left(\dfrac{\psi(x)}{x}\right)^{p}\right)\\
&=&\frac{{\pi}p\,x^{p-1}}{4}\left(\dfrac{t}{\E(t)-(1-t)\K(t)}\dfrac{1-t}{(1+t)(\E'(t)-t\K'(t))}\left(\dfrac{\psi(t)}{t}\right)^{p}\right.\\
& &\qquad\left.-\dfrac{x}{\E(x)-(1-x)\K(x)}\dfrac{1-x}{(1+x)(\E'(x)-x\K'(x))}\left(\dfrac{\psi(x)}{x}\right)^{p}\right)
\end{eqnarray*}
which is positive by Lemma \ref{mylemma}(2),(7) and Theorem \ref{mythm1} since $t>x$ and $p>0$. Hence $f$ is strictly increasing and $f(x)<f(y)=0$. This implies that
$$\psi\left(\left(\frac{x^p+y^p}{2}\right)^{1/p}\right)\leq\left(\frac{\psi(x)^p+\psi(y)^p}{2}\right)^{1/p}.$$

For the case of $p\leq-1$, by previous calculation we have
$$f'(x)
=\dfrac{p}2x^{p-1}\left(\left(\dfrac{\psi(t)}{t}\right)^{-2}\psi'(t)\left(\dfrac{\psi(t)}{t}\right)^{p+1}
-\left(\dfrac{\psi(x)}{x}\right)^{-2}\psi'(x)\left(\dfrac{\psi(x)}{x}\right)^{p+1}\right).$$
Since $(\psi(x)/x)^{p+1}$ is decreasing, we only need to prove $(\psi(x)/x)^{-2}\psi'(x)$ is strictly decreasing in $(0,1)$. In fact,
with the change of variable $x\mapsto(1-t)/(1+t)$,
$$
\left(\dfrac{\psi(x)}{x}\right)^{-2}\psi'(x)=\dfrac{\pi}{4}\left(\dfrac{x(1-x)}{x'(\E-(1-x)\K)}\right)^2=\dfrac{\pi}{4}\left(\dfrac{t'^2}{\E'-t^2\K'}\dfrac{\sqrt{t}}{1+t}\right)^2
$$
which is a product of two positive and strictly increasing functions of $t$ by Lemma \ref{lemma}(1). Hence $f'(x)>0$ and $f$ is strictly increasing in $(0,1)$. Now we have $f(x)<f(y)=0$, and consequently
$$\psi\left(\left(\frac{x^p+y^p}{2}\right)^{1/p}\right)\geq\left(\frac{\psi(x)^p+\psi(y)^p}{2}\right)^{1/p}$$
since $p$ is negative.

The equality case is obvious. This completes the proof.
\end{proof}

%===============================================================================
%===============================================================================
%===============================================================================
\medskip

\section{Applications}
%===============================================================================
In this section we always denote $R=[0,1]\times[0,b]$. Let $\Gamma_b$ and $\Delta_b$ be the families of curves joining the opposite sides of length $b$ of the rectangle $R$, in the exterior and interior of the rectangle, respectively. It is well-known that $\mathcal{M}(\Delta_b)=b$. By the formula of Duren and Pfaltzgraff (\ref{dpformula}), we have
$$\mathcal{M}(\Gamma_b)=\dfrac{1}{\pi}\mu(\psi^{-1}(1/b)).$$
Setting $r=\sqrt{2}-1$ in (\ref{idt2}), we get $\K'(3-2\sqrt{2})=2\K(3-2\sqrt{2})$. By Corollary \ref{specvalu},
$$\mathcal{M}(\Gamma_1)=\dfrac{1}{\pi}\mu(\psi^{-1}(1))=1=\mathcal{M}(\Delta_1).$$
Now we will study the behavior of the modulus $\mathcal{M}(\Gamma_b)$ with respect to the sides of length $b$. The following Theorem \ref{modcomp} shows
\beq
\left\{\begin{array}{ll}
\mathcal{M}(\Gamma_b)>\mathcal{M}(\Delta_b), &\mbox{for} \quad 0<b<1,\\
\mathcal{M}(\Gamma_b)<\mathcal{M}(\Delta_b), &\mbox{for} \quad b>1.
\end{array}\right.
\eeq

\begin{theorem}\label{modcomp}
There exists a number $r_0=8.24639\ldots$ such that the function
$$f(r)=\dfrac{1}{\pi}\mu(\psi^{-1}(r))-\frac{1}{r}$$
is strictly increasing in $(0,r_0)$ and decreasing in $(r_0,\infty)$, with the limiting value $f(\infty)=0$. In particular,
$$\dfrac{1}{\pi}\mu(\psi^{-1}(r))<\frac{1}{r}, \quad \mbox{for} \quad 0<r<1,$$
and
$$\dfrac{1}{\pi}\mu(\psi^{-1}(r))>\frac{1}{r}, \quad \mbox{for} \quad r>1.$$
\end{theorem}

\begin{proof}
Let $s=\psi^{-1}(r)$. Then $r=\psi(s)$ and, by the derivative formula (\ref{psider}),
$$\dfrac{ds}{dr}=(\dfrac{dr}{ds})^{-1}=\dfrac{s'^2}{\pi}\left(\dfrac{\E'(s)-s\K'(s)}{1-s}\right)^2.$$
By differentiation and
$$\dfrac{d\mu(s)}{ds}=\dfrac{-\pi^2}{4ss'^2\K(s)^2},$$
we have
\begin{eqnarray*}
f'(r)&=&\dfrac{1}{\pi}\dfrac{d\mu}{ds}\dfrac{ds}{dr}+\dfrac{1}{r^2}\\
&=&\dfrac{1}{r^2}-\dfrac{1}{4s\K(s)^2}\left(\dfrac{\E'(s)-s\K'(s)}{1-s}\right)^2\\
&=&\dfrac{1}{\psi(s)^2}\left(1-\left(\dfrac{\E(s)-(1-s)\K(s)}{\sqrt{s}(1-s)\K(s)}\right)^2\right).
\end{eqnarray*}
which is positive in $(0,r_0)$ and negative in $(r_0,\infty)$ with $r_0=\psi(0.479047\ldots)=8.24639\ldots$ by Lemma \ref{mylemma}(8).
Hence $f$ is strictly increasing in  $(0,r_0)$ and decreasing in $(r_0,\infty)$. Since $f(1)=0$ and $f(\infty)=0$,
we have $f(r)<0$ for $r\in(0,1)$ and $f(r)>0$ for $r\in(1,\infty)$.
\end{proof}

The next theorem shows that the modulus $\mathcal{M}(\Gamma_b)$ has a logarithmic growth with respect to the length of side $b$.

\begin{theorem}
For $b\in(0,\infty)$,
\beq\label{modulusest}
L(b)<\mathcal{M}(\Gamma_b)<U(b),
\eeq
where
\begin{eqnarray}\label{lb4mod}
L(b)&:=&\dfrac{2}{\pi}\left(1-\left(1+\sqrt{4b/\pi}\right)^{-4}\right)^{1/4}\,\log\left(2\left(1+\sqrt{{4b}/{\pi}}\right)\right)\nonumber\\
&>&\dfrac{2}{\pi}\left(1-\left(1+\sqrt{4b/\pi}\right)^{-1}\right)\,\log\left(2\left(1+\sqrt{{4b}/{\pi}}\right)\right),
\end{eqnarray}
and
\begin{eqnarray}\label{ub4mod}
U(b)&:=&\dfrac{1}{\pi}\log\left(2\left(1+\sqrt{\pi b}\right)^2\left(1+\sqrt{1-\left(1+\sqrt{\pi b}\right)^{-4}}\right)\right)\nonumber\\
&<&\dfrac{2}{\pi}\log\left(2\left(1+\sqrt{\pi b}\right)\right).
\end{eqnarray}

\end{theorem}

\begin{proof}
By Theorem \ref{mythm2} we have
$$\left(\dfrac{\sqrt{r}}{\sqrt{\pi}+\sqrt{r}}\right)^2<s=\psi^{-1}(r)<\left(\dfrac{\sqrt{r}}{\sqrt{4/\pi}+\sqrt{r}}\right)^2,\quad r\in(0,\infty).$$
By \cite[Theorem 5.13(4),(5)]{avv},
$$\sqrt{s'}\log{\dfrac{4}{s}}<\mu(s)<\log{\dfrac{2(1+s')}{s}},\quad s\in(0,1).$$
Combining the above inequalities and replacing $r$ with $1/b$, we get the inequalities (\ref{modulusest}).
The inequality (\ref{lb4mod}) follows from the inequality $1-a^x>(1-a)^x$ for $a\in(0,1)$ and $x\in(1,\infty)$.
The inequality (\ref{ub4mod}) is obvious.
\end{proof}

\medskip
\begin{figure}[h]
\begin{minipage}[t]{0.45\linewidth}
\centering
\includegraphics[width=8cm]{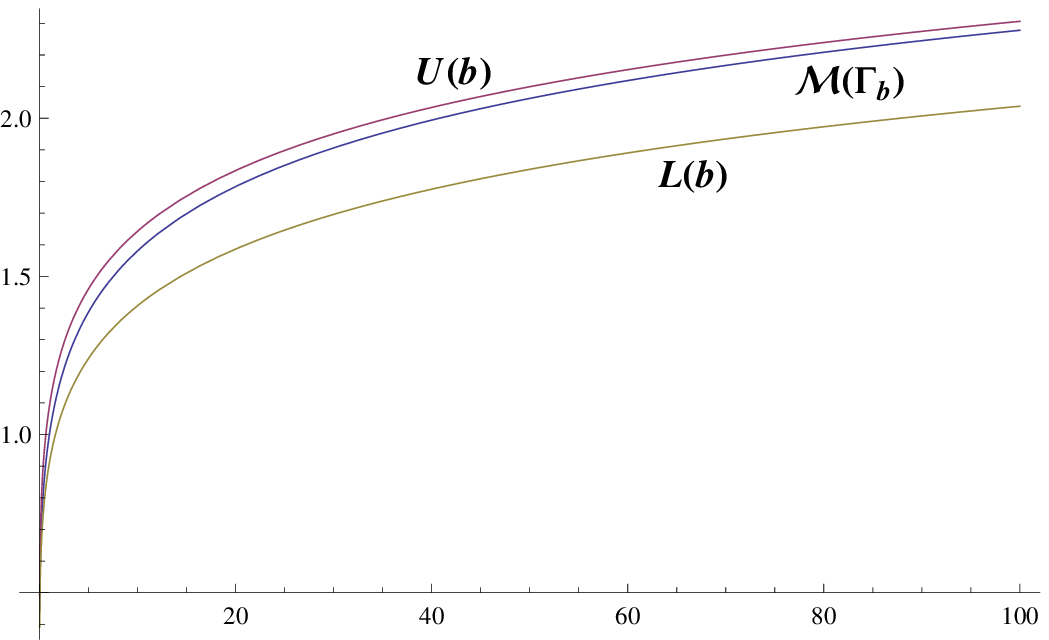}
%\caption{ \label{h2}}
\end{minipage}
\end{figure}

\begin{theorem}
For $a,b\in(0,\infty)$,
\begin{enumerate}
\item $\mathcal{M}(\Gamma_{2ab/(a+b)})\leq\sqrt{\mathcal{M}(\Gamma_a)\mathcal{M}(\Gamma_b)}\leq
    \dfrac{\mathcal{M}(\Gamma_a)+\mathcal{M}(\Gamma_b)}{2}\leq\mathcal{M}(\Gamma_{(a+b)/2})$;\\
\item $\left\{\begin{array}{ll}
\mathcal{M}(\Gamma_{H_p(a,b)})\leq H_p(\mathcal{M}(\Gamma_a),\mathcal{M}(\Gamma_b)),&p\leq-1,\vspace{1mm}\\
\mathcal{M}(\Gamma_{H_p(a,b)})\geq H_p(\mathcal{M}(\Gamma_a),\mathcal{M}(\Gamma_b)),&p\geq1.\\
\end{array}\right.$
\end{enumerate}
Equality holds in each case if and only if $a=b$.
\end{theorem}

\begin{proof}
In part (1), the second inequality is clear. For the third inequality, let $s=\psi^{-1}(1/a)$, $t=\psi^{-1}(1/b)$. Then
\begin{eqnarray*}
\dfrac{\mathcal{M}(\Gamma_a)+\mathcal{M}(\Gamma_b)}{2}&=&\dfrac{1}{\pi}\dfrac{\mu(s)+\mu(t)}{2}\\
&\leq&\dfrac{1}{\pi}\mu(\sqrt{st})\leq\dfrac{1}{\pi}\mu(H_{-1}(s,t))\\
&\leq&\dfrac{1}{\pi}\mu(\psi^{-1}(H_{-1}(\psi(s),\psi(t))))\\
&=&\dfrac{1}{\pi}\mu(\psi^{-1}(H_{-1}(1/a,1/b)))\\
&=&\mathcal{M}(\Gamma_{(a+b)/2}),
\end{eqnarray*}
where the first inequality follows from \cite[Theorem 5.12(1)]{avv} (also see \cite[Theorem]{wzj}) and the third inequality follows from  Theorem \ref{mythm5}.
Let $m(a)=\mu(\psi^{-1}(a))/\pi$ and $u=\psi^{-1}(a)$. By logarithmic differentiation, we have
$$\dfrac{d}{da}\log m(a)=-\dfrac{1}{2u\K'(u)\K(u)}\left(\dfrac{\E'(u)-u\K'(u)}{1-u}\right)^2,$$
which is strictly increasing in $u$ by Lemma \ref{lemma}(2) and Lemma \ref{mylemma}(4),  and hence strictly increasing in $a$. Since $m(a)$ is logarithmic convex,
we have
$$m\left(\dfrac{a+b}{2}\right)\leq\sqrt{m(a)m(b)},$$
which implies the first inequality in part (1) by replacing $a,b$ with $1/a,1/b$, respectively.

For the part (2), let $M(x):=\mathcal{M}(\Gamma_x)$.
Let $0<x<y<1$ and $t=((x^p+y^p)/2)^{1/p}>x$. Define
$$f(x)=M(t)^p-\dfrac{M(x)^p+M(y)^p}{2}.$$
By differentiation, we have $dt/dx=\frac{1}{2}(x/t)^{p-1}$ and
\beq\label{diffmod}
f'(x)=\dfrac{p}2x^{p-1}\left(\left(\dfrac{M(t)}{t}\right)^{p-1}M'(t)-\left(\dfrac{M(x)}{x}\right)^{p-1}M'(x)\right).
\eeq
Let $M(x)=m(a)$, then $a=1/x=\psi(u)$. Now we have
\begin{eqnarray*}
\left(\dfrac{M(x)}{x}\right)^{p-1}M'(x)&=&(m(a)a)^{p-1}m'(a)(-a^2)\\
&=&\left(\dfrac{\mu(u)\psi(u)}{\pi}\right)^{p-1}\psi(u)^2\dfrac{1}{4u\K(u)^2}\left(\dfrac{\E'(u)-u\K'(u)}{1-u}\right)^2\\
&=&\left(\dfrac{\mu(u)\psi(u)}{\pi}\right)^{p-1}\left(\dfrac{\E(u)-(1-u)\K(u)}{\sqrt{u}(1-u)\K(u)}\right)^2,
\end{eqnarray*}
which is strictly increasing in $u$ by Lemmas \ref{mylemmupsi} and \ref{mylemma}(8), and hence strictly decreasing in $x$ for each $p\geq1$. This implies that $f'(x)<0$ if $p\geq1$.

For the case of $p\leq-1$, we have
\begin{eqnarray*}
\left(\dfrac{M(x)}{x}\right)^{p-1}M'(x)&=&(m(a)a)^{p-1}m'(a)(-a^2)=-(m(a)a)^{p+1}m(a)^{-2}m'(a)\\
&=&\left(\dfrac{\mu(u)\psi(u)}{\pi}\right)^{p+1}\dfrac{1}{u\K'(u)^2}\left(\dfrac{\E'(u)-u\K'(u)}{1-u}\right)^2,
\end{eqnarray*}
which is strictly decreasing in $u$ by Lemmas \ref{mylemmupsi}, \ref{lemma}(2) and \ref{mylemma}(4), and hence strictly increasing in $x$ for each $p\leq-1$. Since $p$ is negative, this still implies that $f'(x)<0$.

It is easy to see that $f'(x)<0$ implies the inequalities in the part (2).
\end{proof}

\begin{openprob}
What is the exact domain of $p$ for which the function $\psi$ is $H_{p,p}$-convex (concave)? More generally, find the exact $(p,q)$ domain for which the function $\psi$ is $H_{p,q}$-convex (concave). The same questions can be asked for the modulus $\mathcal{M}(\Gamma_b)$.
\end{openprob}

\medskip

\subsection*{Acknowledgments}
The research of Matti Vuorinen was supported by the Academy of Finland, Project 2600066611.
Xiaohui Zhang is indebted to the CIMO (Grant TM-09-6629) and the Finnish National Graduate School of Mathematics and
its Applications for financial support. Both authors wish to thank \'Arp\'ad Baricz and the referee for their helpful comments on the manuscript.

%The authors are
%indebted to the referee for his/her constructive comments and
%helpful suggestions, which improved the first draft of this paper.

%===============================================================================
%===============================================================================
%===============================================================================
%===============================================================================
%==============================================================================

\end{document}